\documentclass[12pt]{amsart}
\usepackage{amssymb, enumerate, mdwlist}

\evensidemargin\paperwidth
\advance\evensidemargin-\textwidth
\advance\oddsidemargin-1in
\evensidemargin\oddsidemargin

\topmargin\paperheight
\advance\topmargin-\textheight
\advance\topmargin-1in

\theoremstyle{plain}
\newtheorem{theorem}{Theorem}[section]
\theoremstyle{plain}

\theoremstyle{plain}
\newtheorem{lemma}[theorem]{Lemma}
\theoremstyle{plain}

\theoremstyle{plain}

\theoremstyle{plain}
\newtheorem{definition}[theorem]{Definition}
\theoremstyle{plain}

\theoremstyle{remark}

\theoremstyle{remark}
\newtheorem{example}[theorem]{Example}
\theoremstyle{remark}

\title
[Property A for coarse spaces]
{Property A for coarse spaces}
\author{Hiroki Sako}
\thanks{The author is a Research Fellow of the Japan Society for the Promotion of Science (PD)}
\address
{Research Institute for Mathematical Sciences, Kyoto University, Kyoto 606-8502, Japan}
\email
{hiroki.sako@gmail.com}

\begin{document}

\begin{abstract}
Property A introduced by Guoliang Yu
is an amenability-type property for metric spaces.
In this article, we study property A for uniformly locally finite coarse spaces.
Main examples of coarse spaces are
a metric space,
a set equipped with a discrete group action,
and a sequence of finitely generated groups.
The purpose of this article is to give complete proofs to
related basic facts.
\end{abstract}


\maketitle

\section{Introduction}
The subject of this article is amenability for generalized 
metric spaces.
Property A for discrete metric spaces was introduced by G. Yu
in his study of coarse Baum--Connes conjecture \cite{Yu:CoarseHilbert}.
Property A is widely recognized as an amenability-type condition.
For discrete groups,
amenability depends only on its large scale structure.
Property A is also independent of local features of
metric spaces, so it is natural to
consider a new concept of spaces. 
J. Roe introduced a notion called coarse space in \cite{RoeLectureNote}.
The category of coarse spaces encompasses
\begin{itemize}
\item
a metric space;
\item
a discrete group;
\item
a set equipped with a discrete group action
(Example \ref{ExampleGroupCoarse});
\item
a sequence of finite generated groups with fixed generators
(Example \ref{ExampleGroupSeq}).
\end{itemize}
In this article, we clarify the definition of property A for uniformly locally finite coarse spaces,
which is not necessarily metrizable.
Uniform local finiteness of a coarse space $X$ means that $X$ has very poor local structure.
We prove that the property can be rephrased in the following context:
\begin{itemize}
\item
a characterization by Hilbert space;
\item
nuclearity of the uniform Roe algebra;
\item
the operator norm localization property (ONL).
\end{itemize}

This article is not a usual research paper.
Ideas in the proofs have already appeared in papers which dealt with property A for metric spaces
(\cite{PaperBrodzkiNibloWright}, \cite{ONLPoriginal}, \cite{HigsonRoe}, \cite{SkandalisTuYu}, \cite{RoeLectureNote},
\cite{SakoONLP}).
We give a precise proof for the equivalence between property A
and ONL, which has been proved only for metric spaces.
The general form (Theorem \ref{TheoremONLP}) of the theorem plays a key role in the paper \cite{SakoNote}, because the argument is not closed in the category of metric spaces.
The framework of coarse spaces is not only necessary but also appropriate.

\section{Definition of Property A}

\subsection{Coarse space}
A coarse space is a set $X$ equipped with a coarse structure $\mathcal{C}$. The coarse structure $\mathcal{C}$ is a family of
subsets of $X^2$ and satisfies several requirements. 
To describe them, we will use the following notations.
For subsets $T, T_1, T_2 \subset X^2$, the inverse $T^{-1}$ and the product $T_1 \circ T_2$
are defined as follows:
\begin{eqnarray*}
T^{-1} &=& \{(x, y) \in X^2 \ ;\ (y,x) \in T\},\\
T_1 \circ T_2 &=& \{ (x, y) \in X^2  
\ ;\ \exists z \in X, (x, z) \in T_1, (z, y) \in T_2 \}.
\end{eqnarray*}
Denote by $T^{\circ n}$ the $n$-th power 
$T \circ T \circ \cdots \circ T$.
In the case that the subsets are graphs of partially defined maps,
these notations coincide with the usual conventions for mappings.
For subsets $Y \subset X$ and $T \subset X^2$,
let $T[Y]$ be the subset of $X$ defined by
\begin{eqnarray*}
\{x \in X\ ; \ \exists y \in Y, (x, y) \in T \}.
\end{eqnarray*}
For a singleton $\{x\}$, we simply write $T[x]$ for $T[\{ x \}]$. 
A subset $Y \subset X$ is called a $T$-bounded set if there exists $x \in X$ such that $Y \subset T[x]$.
If $Y$ is a $T_2$-bounded set, then $T_1[Y]$ is a $(T_1 \circ T_2)$-bounded set.

\begin{definition}
[Definition 2.3 in \cite{RoeLectureNote}]
Let $X$ be a set.
A family $\mathcal{C}$ of subsets of $X^2$ is 
called a coarse structure on X if 
\begin{itemize}
\item
$\Delta_X = \{(x, x) ; x \in X\} \in \mathcal{C}$;
\item
If $T \in \mathcal{C}$, then $T^{-1} \in \mathcal{C}$;
\item
If $T_1, T_2 \in \mathcal{C}$, 
then $T_1 \circ T_2 \in \mathcal{C}$;
\item
If $T_1, T_2 \in \mathcal{C}$, 
then $T_1 \cup T_2 \in \mathcal{C}$;
\item
If $T_1 \in \mathcal{C}$ and $T_2 \subset T_1$, 
then $T_2 \in \mathcal{C}$.
\end{itemize}
The pair $(X, \mathcal{C})$ is called a coarse space. 
Elements of $\mathcal{C}$ are called controlled sets or entourages.
\end{definition}
Two elements $x, y \in X$ are said to be connected if $\{(x, y)\} \in \mathcal{C}$. 
If arbitrary two points are connected, the space $X$ is said to be connected.

\begin{example}\label{ExampleMetricSpace}
Let $(X, d)$ be a metric spaces.
A coarse structure $\mathcal{C}$ is defined 
on $X$ by
$\mathcal{C} = \left\{ T \subset X^2 \ ;\ \exists S>0,
d |_{T} \le S \right\}$.
The space $X$ is connected.
\end{example}

\begin{example}\label{ExampleGroupCoarse}
Let $X$ be a set on which a discrete group $G$ acts.
For a finite subset $K \subset G$, let $T_K$ be the orbit of $K$.  That is, $T_K = \{ (k x, x) \in X^2 ; k \in K, x \in X\}$.
The coarse structure $\mathcal{C}_G$ on $X$ is the collection $\{T; \exists \textrm{ finite } K \subset G, T \subset T_K\}$.
If $G$ is countable and the action on $X$ is transitive, then the coarse structure is realized by some metric.
The group $G$ naturally has a coarse structure defined by the left transformation action of $G$.
\end{example}

\begin{example}\label{ExampleGroupSeq}
Let $\left\{\left(G^{(m)}, g^{(m)}_1, g^{(m)}_2, \cdots, g^{(m)}_k\right)\right\}_{m = 1}^\infty$
be a sequence of groups with fixed $k$-generators.
Define surjective homomorphisms $\phi^{(m)} \colon F_k \rightarrow G^{(m)}$ from the free group 
$F_k = \langle g_1, \cdots, g_k \rangle$ by $\phi^{(m)}(g_i) = g_i^{(m)}$.
The set $\bigsqcup_{m = 1}^\infty G^{(m)}$ is equipped with an $F_k$-action defined by
$h \cdot g = \phi^{(m)}(h) g, g \in G^{(m)}$. 
This action gives a coarse structure on the disjoint union of groups $\bigsqcup_{m = 1}^\infty G^{(m)}$.
This coarse space is not connected.
\end{example}

\begin{definition}
A coarse space $(X, \mathcal{C})$ is said to be uniformly locally finite
if every controlled set $T \in \mathcal{C}$ satisfies the
inequality $\sup_{x \in X} \sharp(T[x]) < \infty$.
\end{definition}

The coarse spaces in Example \ref{ExampleGroupCoarse}
and Example \ref{ExampleGroupSeq} are uniformly locally finite.
In many references, a metric space whose coarse structure is uniformly locally finite is called
a metric space with bounded geometry.

\subsection{Definition of property A}
The definition of property A for $X$ is given by a F{\o}lner
condition of $X \times \mathbb{N}$.

\begin{definition}[Definition 2.1 of Yu \cite{Yu:CoarseHilbert}]
A discrete metric space $X$ 
is said to have property A
if for every $\epsilon > 0$ and $R > 0$, there exist
$S > 0$ and a family of finite subsets 
$\{A_x\}_{x \in X}$ of $X \times \mathbb{N}$ such that
\begin{itemize}
\item
$A_x \subset \{y \in X; d(x, y) < S\} \times \mathbb{N}$;
\item
$(x, 1) \in A_x$;
\item
The symmetric difference 
$A_x \bigtriangleup A_y = (A_x \setminus A_y) \cup (A_y \setminus A_x)$ satisfies the inequality $\sharp(A_x \bigtriangleup A_y) 
< \epsilon \sharp(A_x \cap A_y)$, 
when $d(x, y) \le R$.
\end{itemize}
\end{definition}

We define property A for coarse spaces as follows.
\begin{definition}\label{DefinitionPropertyA}
A uniformly locally finite coarse space $(X, \mathcal{C})$ 
is said to have property A
if for every positive number $\epsilon$ and every controlled set $T \in \mathcal{C}$, 
there exist
a controlled set $S \in \mathcal{C}$ 
and a subset 
$A \subset S \times \mathbb{N}$ such that
\begin{itemize}
\item
For $x \in X$, $A_x = \{(y, n) \in X \times \mathbb{N}\ 
; (x, y, n) \in A\}$ is finite;
\item
$\Delta_X \times \{1\} \subset A$,
where $\Delta_X$ is the diagonal subset of $X^2$;
\item
$\sharp(A_x \bigtriangleup A_y) 
< \epsilon \sharp(A_x \cap A_y)$, 
if $(x, y) \in T$.
\end{itemize}
\end{definition}
The second condition can be replaced with the following:
\begin{itemize}
\item
For $x \in X$, $A_x$ is not empty.
\end{itemize}
This is a conclusion of the proof of Theorem \ref{TheoremHulanicki-type}.

\section{Characterizations by Hilbert spaces and positive definite kernels}
Let $X$ be a uniformly locally finite coarse space.
For a controlled set $T \subset X^2$ of $X$,
we define a linear space $E_T$ of bounded linear operators on $\ell^2(X)$ as follows:
\[E_T
=
\{a \in \mathbb{B}(\ell^2 (X)); 
\langle a \delta_y, \delta_x \rangle = 0 \textrm{\ if\ }
(x, y) \in X^2 \setminus T\}.
\]
We note that
\begin{itemize}
\item
if $b_1 \in E_{T_1}$ and $b_2 \in E_{T_2}$,
then $b_1 b_2 \in E_{T_1 \circ T_2}$;
\item
and that if $b \in E_T$, then $b^* \in E_{T^{-1}}$.
\end{itemize}
We also define a C$^*$-algebra called the uniform Roe algebra $C^*_\mathrm{u} (X)$ of $X$, which plays a key role for characterizations of property A. The algebra $C^*_\mathrm{u} (X)$ is defined by
$\overline{\bigcup_{T: \textrm{controlled}} E_T}$.

\begin{theorem}\label{TheoremHulanicki-type}
For a uniformly locally finite coarse space $(X, \mathcal{C})$,
the following conditions are equivalent:
\begin{enumerate}
\item\label{ConditionPropertyA}
$X$ has property A;
\item\label{ConditionEll2}
For every $\epsilon > 0$,
and every $T \in \mathcal{C}$,
there exists a map $\eta \colon X \rightarrow \ell^2(X)$
assigning unit vectors such that
\begin{itemize}
\item
if $(x, y) \in T$, then $\|\eta_x - \eta_y\| < \epsilon$,
\item
$\{(x, y) \in X^2 ; y \in \mathrm{supp}(\eta_x)\}$ is controlled;
\end{itemize}
\item\label{ConditionHilbert}
For every $\epsilon > 0$,
and every $T \in \mathcal{C}$,
there exist a Hilbert space $\mathcal{H}$
and a map $\eta \colon X \rightarrow \mathcal{H}$
assigning unit vectors such that
\begin{itemize}
\item
if $(x, y) \in T$, then $\|\eta_x - \eta_y\| < \epsilon$,
\item
$\{(x, y) \in X^2 ; \langle \eta_y, \eta_x \rangle \neq 0\}$ is controlled;
\end{itemize}
\item\label{ConditionPositiveKernel}
For every $\epsilon > 0$ and every $T \in \mathcal{C}$, there exists and a positive definite kernel $k \colon X^2 \rightarrow \mathbb{C}$ such that
\begin{itemize}
\item
if $(x, y) \in T$, then $\|1 - k(x, y)\| < \epsilon$,
\item
$\{(x, y) \in X^2 ; k(x, y) = 0\}$ is controlled.
\end{itemize}
\end{enumerate}
\end{theorem}

For basic facts of positive definite kernel (or function of positive type), see Appendix C of the book \cite{BekkaDelaHarpeValette}.
This theorem for metric spaces is given in 
\cite[Theorem 3]{PaperBrodzkiNibloWright}.

\begin{proof}
Suppose that $X$ has property A.
Let $\epsilon$ be an arbitrary positive number and 
let $T$ be a controlled set of $X$.
Then there exist a controlled set $S$ and a subset $A \subset S \times \mathbb{N}$ which satisfy the conditions of Definition \ref{DefinitionPropertyA}.
Define a subset $A_x(y)$ of $\mathbb{N}$ by
\[A_x(y) = \{n \in \mathbb{N} ; (y, n) \in A_x\}
= \{n \in \mathbb{N} ; (x, y, n) \in A \}.
\]
Let $\|\cdot\|_1$ denote the norm of $\ell^1(X)$.
For $x \in X$, define vectors $\zeta_x, \xi_x$ of $\ell^1(X)$
by
\[\zeta_x(y) = \sharp(A_x(y)), \quad \xi_x 
= \zeta_x / \|\zeta_x\|_1.\]
For $(x, y) \in T$, we have
\begin{eqnarray*}
\|\xi_x - \xi_y\|_1 
&=& \frac{\left\|\|\zeta_y\|_1 \zeta_x - \|\zeta_x\|_1 \zeta_y\right\|_1}
{\| \zeta_x \|_1 \|\zeta_y \|_1}\\
&\le& 
\frac{|\|\zeta_y\|_1 - \|\zeta_x\|_1| \cdot \|\zeta_x\|_1 + 
\|\zeta_x\|_1 \cdot \|\zeta_x - \zeta_y\|_1}
{\| \zeta_x \|_1 \|\zeta_y \|_1}\\
&\le& 
2 \frac{\|\zeta_x - \zeta_y\|_1}
{\|\zeta_y \|_1}\\
&\le& 
2 \frac{\sharp(A_x \bigtriangleup A_y)}
{\sharp(A_y)} \\
&<& 2 \epsilon.
\end{eqnarray*}
Define a unit vector $\eta_x$ of $\ell^2(X)$ by
$\eta_x(y) = \sqrt{\xi_x(y)}.$ Then we have
\[\|\eta_x - \eta_y\|_2 \le \sqrt{\|\xi_x - \xi_y \|_1} 
< \sqrt{2 \epsilon}, \quad (x, y) \in T.\]
The set $\{(x, y) \in X^2 ; y \in \mathrm{supp}(\eta_x)\}$ is controlled, since it is included in $S$.
Here we obtain condition (\ref{ConditionEll2}).

When we denote by $S$ the subset
$\{(x, y) \in X^2 ; y \in \mathrm{supp}(\eta_x)\}$,
the subset
$\{(x, y) \in X^2 ; \langle \eta_y, \eta_x \rangle \neq 0\}$
is included in $S \circ S^{-1}$.
Therefore (\ref{ConditionEll2}) implies
(\ref{ConditionHilbert}).

Suppose that (\ref{ConditionHilbert}) holds.
For arbitrary $\epsilon>0$ and 
$T \in \mathcal{C}$,
there exists a map $\eta \colon X \rightarrow \mathcal{H}$
such that $\|\eta_x\| = 1$ and
\begin{itemize}
\item
if $(x, y) \in T$, then $\|\eta_x - \eta_y\| < \epsilon$,
\item
$\{(x, y) \in X^2 ; \langle \eta_x, \eta_y \rangle \neq 0\}$ is controlled.
\end{itemize}
Define a function $k \colon X^2 \rightarrow \mathbb{C}$
by $k(x, y) = \langle \eta_x, \eta_y \rangle$.
The function $k$ satisfies condition (\ref{ConditionPositiveKernel}).

For the converse direction, we first prove that condition (\ref{ConditionPositiveKernel}) implies condition (\ref{ConditionEll2}).
Suppose that condition (\ref{ConditionPositiveKernel}) holds. 
For arbitrary $\epsilon > 0$ and arbitrary controlled set $T$,
take a positive definite kernel 
$k : X^2 \rightarrow \mathbb{C}$ satisfying (\ref{ConditionPositiveKernel}).
Then there exists an operator $a$ in $C^*_\mathrm{u} (X)$ such that $\langle a \delta_y, \delta_x \rangle = k(x, y)$.
Note that the operator $a$ is positive.
So there exist $S \in \mathcal{C}$ and $b \in E_S$ which satisfies $\|a - b^* b\| < \epsilon$. 
For $x \in X$, we define a vector $\zeta_x \in \ell^2(X)$
by $b \delta_x$. Since $\{(x, y) \in X^2 ; y \in \mathrm{supp}(\zeta_x)\}$ is included in $S^{-1}$, it is controlled.
For $(x, y) \in T$, we have
\begin{eqnarray*}
\|\zeta_x - \zeta_y\|^2 
&=&
k(x, x) - 2 \mathrm{Re}(k(x, y)) + k(y, y)
< 4 \epsilon,\\
\left| 1 - \| \zeta_x \|^2 \right|
&=& |1 - k(x, x)| < \epsilon.
\end{eqnarray*}
When we define $\eta_x$ by $\zeta_x / \|\zeta_x\|$,
$\|\eta_x - \eta_y\|$ is uniformly close to $0$ for $(x, y) \in T$.
Thus we obtain condition (\ref{ConditionEll2}).

Finally, we prove that condition (\ref{ConditionEll2})
implies condition (\ref{ConditionPropertyA}).
For arbitrary $\epsilon > 0$ and $T \in \mathcal{C}$,
take unit vectors $\{\eta_x\}_{x \in X}$ satisfying
condition (\ref{ConditionEll2}).
We define a positive element $\xi_x$ of $\ell^1(X)$ by
$\xi_x(y) = |\eta_x (y)|^2$.
For $(x, y) \in T$, we have
$\| \xi_x - \xi_y \|_1 
\le \| \eta_x + \eta_y \|_2 \| \eta_x - \eta_y \|_2 
< 2 \epsilon.$
The cardinality of the support $\mathrm{supp}(\xi_x)$ 
is uniformly bounded.
More precisely, 
$\sup_{x \in X} \sharp(\mathrm{supp}(\xi_x)) < \infty$.
By replacing $\xi_x$, we obtain the following:
there exists a natural number $m$ such that
\begin{itemize}
\item
$|1 - \|\xi_x\|_1| < \epsilon$;
\item
If $(x, y) \in T$, then $\|\xi_x - \xi_y\|_1 < \epsilon$;
\item
$\{(x, y) \in X^2 ; y \in \mathrm{supp}(\xi_x)\}$ is controlled;
\item
For every $(x, y) \in X^2$,
$\xi_x(y) \in \{n / m ; n \in \mathbb{N} \cup \{0\} \}$;
\item
$\xi_x(x) \neq 0$.
\end{itemize}
Let $S$ denote the controlled set $\{(x, y) \in X^2 ; y \in \mathrm{supp}(\xi_x)\}$.
Define $A \subset S \times (\mathbb{N} \cup \{0\})$ by
\[A = \{(x, y, n) ; n < m \cdot \xi_x(y)\}.\]
Then the family of the finite sets $\{A_x = \{(y, n) ; (x, y, n) \in A \}\}_{x \in X}$ satisfy
\[\sharp(A_x \bigtriangleup A_y) =  
\|m \xi_x - m \xi_y\|_1
< \epsilon m < \frac{\epsilon}{1 - \epsilon} \sharp(A_x), 
\quad (x, y) \in T.\]
We also have
\begin{itemize}
\item
$\sharp(A_x \bigtriangleup A_y) 
< 2 \epsilon \sharp(A_x \cap A_y) /(1 - \epsilon), 
\ (x, y) \in T$;
\item
$(x, 0) \in A_x$.
\end{itemize}
It follows that $X$ has property A.
\end{proof}

\section{Characterization by translation C$^\ast$-algebras}
We generalize the easier half of Theorem 5.3 in \cite{SkandalisTuYu} in this section.
\begin{lemma}\label{LemmaNorm}
Let $X$ be a set and let $T$ be a subset of $X^2$.
Suppose that $\sharp(T[x]), \sharp(T^{-1} [x])$
are uniformly bounded.
Let $b$ be an operator whose matrix coefficients are zero on $X^2 \setminus T$.
Then we have the following estimate of the operator norm:
\[ \|b\| \le \max 
\left\{ \sup_{x \in X} \sharp(T[x]), \sup_{x \in X} \sharp(T^{-1} [x]) \right\} \cdot 
\sup_{x, y \in X} |\langle b \delta_y, \delta_x \rangle|.\]
\end{lemma}
This lemma is one version of the Schur tests.
\begin{proof}
We denote by $b_{x, y}$ the matrix coefficient $\langle b \delta_y, \delta_x \rangle$.
Let $\xi = (\xi_x)_{x \in X}$ and $\eta = (\eta_x)_{x \in X}$
be elements of $\ell^2(X)$.
By the Cauchy--Schwarz inequality, we have
\begin{eqnarray*}
|\langle b \eta, \xi \rangle|
&=&
\left|\sum_{x, y} \overline{\xi_x} b_{x, y} \eta_y \right|\\
&\le&
\sum_{x, y} |\xi_x| |b_{x, y}|^{1/2} |b_{x, y}|^{1/2} |\eta_y|\\
&\le&
\left(\sum_{x, y} |\xi_x|^2 |b_{x, y}| \right)^{1/2} 
\left(\sum_{x, y} |b_{x, y}| |\eta_y|^2 \right)^{1/2}\\
&\le&
\sup_{x, y} |b_{x, y}|
\left(\sup_{x}\sharp(T^{-1}[x]) \sum_{x} |\xi_x|^2 \right)^{1/2} 
\left(\sup_{y}\sharp(T[y]) \sum_{y} |\eta_y|^2 \right)^{1/2}\\
&\le&
\max \left\{ \sup_{y \in X} \sharp(T[y]), \sup_{x \in X} \sharp(T^{-1} [x]) \right\} \cdot
\sup_{x, y} |b_{x, y}| \cdot
\|\xi\| \cdot \|\eta\|.
\end{eqnarray*}
Thus we obtain the inequality in the lemma.
\end{proof}

\begin{theorem}\label{TheoremNuclear}
Let $X$ be a uniformly locally finite coarse space.
If $X$ has property A, then the translation C$^\ast$-algebra $C^*_\mathrm{u}(X)$ is nuclear.
\end{theorem}
The proof is given by a verbatim translation of \cite[Proposition 11.41]{RoeLectureNote}.
For a controlled set $S \subset X^2$,
we denote by $C_S$ the C$^\ast$-algebra 
$\prod_{z \in X} \mathbb{B}(\ell^2(S[z]))$.
We also define a unital completely positive map
$\Phi_S : C^*_\mathrm{u}(X) \rightarrow C_S$ by
\[
\Phi_S(b) = [[b_{x, y}]_{x, y \in S[z]}]_{z \in X},
\]
where $b_{x, y}$ is the matrix coefficient $\langle b \delta_y, \delta_x \rangle$.

\begin{proof}
For any positive number $\epsilon$ and
any finite subset $\mathcal{F}$ of $C^*_\mathrm{u}(X)$, 
we find a controlled set $S$ and construct a unital completely positive map
$\Psi : C_S \rightarrow C^*_\mathrm{u}(X)$ such that
$\|\Psi \circ \Phi_S (b) - b\| < \epsilon$, for $b \in \mathcal{F}$.
Since $C_S$ is nuclear, the construction provides the proof.
By approximating elements of $\mathcal{F}$,
we may assume that for every $b \in \mathcal{F}$ the supports of the matrix coefficients $[b_{x, y} = \langle b \delta_y, \delta_x\rangle]_{x, y \in X}$ are included 
in a controlled set $T$.

Define $\delta> 0$ by 
$\displaystyle{\frac{\epsilon}
{\max_{b \in \mathcal{F}}\|b\| \cdot 
\max\{ \sup \sharp(T[x]), \sup \sharp(T^{-1}[x])\}}}$.
By Theorem \ref{TheoremHulanicki-type},
there exists a map $\eta \colon X \rightarrow \ell^2(X)$
assigning unit vectors such that
\begin{itemize}
\item
If $(x, y) \in T$, then $\|\eta_x - \eta_y\| < \delta$;
\item
$S = \{(x, z) \in X^2 ; z \in \mathrm{supp}(\eta_x)\}$ is controlled.
\end{itemize}
Define $\Psi : C_S \rightarrow C^*_\mathrm{u}(X)$ by
\[
\Psi([[c^{(z)}_{x, y}]_{x, y \in S[z]}]_{z \in X}) =
\left[
\sum_{z \in X} \overline{\eta_x(z)} c^{(z)}_{x, y} \eta_y(z)
\right]_{x, y \in X}.
\]
If $x \notin S[z]$ or if $y \notin S[z]$, we define $c^{(z)}_{x, y}$ by $0$.
It is routine to prove that $\Psi$ is unital and completely positive.
For $b \in \mathcal{F}$, the operator $\Psi \circ \Phi_S (b)$ is expressed as
\[
\left[
\sum_{z \in X} \overline{\eta_x(z)} b_{x, y} \eta_y(z)
\right]_{x, y \in X}
=
\left[
\langle \eta_y, \eta_x \rangle b_{x, y}
\right]_{x, y \in X}.
\]
Therefore the matrix coefficients of $\Psi \circ \Phi_S (b)$
are not zero only on the controlled set $T$.
For $(x, y) \in T$, 
$|1 - \langle \eta_y, \eta_x \rangle | 
= |\langle \eta_y, \eta_y - \eta_x \rangle| < \delta$.
It follows that 
\begin{eqnarray*}
\|\Psi \circ \Phi_S (b) - b\| 
&=& 
\|\left[
(1- \langle \eta_y, \eta_x \rangle) b_{x, y}
\right]_{x, y \in X}\| \\
&\le&
\max\{\sup\sharp(T[x]), \sup\sharp(T^{-1}[x])\}
\cdot \sup_{(x, y) \in T}
|(1- \langle \eta_y, \eta_x \rangle) b_{x, y}|\\
&<& \max\{ \sup \sharp(T[x]), \sup \sharp(T^{-1}[x])\} \cdot \delta \|b\|\\
&<& \epsilon.
\end{eqnarray*}
The second line is due to Lemma \ref{LemmaNorm}.
\end{proof}

The converse is also true.
In the paper
\cite{SakoNote},
the author proves the following stronger claim:
\begin{theorem}
Local reflexivity of  $C^*_\mathrm{u} (X)$
is equivalent to property A of $X$.
\end{theorem}
Local reflexivity is weaker than exactness and nuclearity for general C$^*$-algebras.
As a consequence,
all the following properties of $C^*_\mathrm{u} (X)$
are equivalent:
nuclearity, exactness, and local reflexivity.

\section
{Characterization by the operator norm localization}

Roughly speaking, Theorem \ref{TheoremHulanicki-type}
means that a space $X$ has property A if and only if its large scale structure can be described by a Hilbert space.
The coarse amenability of $X$ can be further rephrased in terms of operator norms.

\subsection{Definitions of the operator norm localization property}
Chen, Tessera, Wang and Yu defined 
the operator norm localization property
in \cite[section 2]{ONLPoriginal}. 
The original definition is given for metric spaces.
For a general uniformly locally finite coarse space $X$, we define the property as follows.
\begin{definition}
A uniformly locally finite coarse space $(X, \mathcal{C})$ is said to have the operator norm localization property (ONL) if 
for every $c < 1$ and $T \in \mathcal{C}$,
there exists a controlled set $S$ satisfying condition $(\beta)$:
for every operator $a \in E_T$,
there exists a unit vector $\eta \in \ell^2 (X)$ such that
$\mathrm{supp}(\eta)$ is an $S$-bounded set and
$c \| a \| \le \| a \eta \|$.
\end{definition}
This definition is analogous to condition $(iv)$ in \cite[Proposition 3.1]{SakoONLP}.
We may replace `for every $c < 1$' with
`there exists $c < 1$.'

\begin{lemma}\label{Lemma quantifier}
A uniformly locally finite coarse space $X$ has ONL, if and only if 
there exists $0 < c < 1$ such that for every controlled set $T$,
there exists a controlled set $S$ satisfying condition $(\beta)$.
\end{lemma}
Following an idea in \cite[Proposition 2.4]{ONLPoriginal},
we give a proof.

\begin{proof}
Assume that $X$ satisfies the property in Lemma \ref{Lemma quantifier}
with respect to a constant $c < 1$.
Let $\kappa$ be an arbitrary real number less than $1$
and let $T$ be an arbitrary controlled set.
Replacing $T$ by $\Delta_X \cup T \cup T^{-1}$, we may assume that $T$ is symmetric and includes $\Delta_X$.
Choose a natural number $n$ satisfying $\kappa^n < c$. 

By the assumption on $X$,
there exists $S$ satisfying the following condition: for every $b \in E_{T^{\circ 2n}}$,
there exists a unit vector $\xi \in \ell^2 (X)$ such that
$\mathrm{supp}(\xi)$ is an $S$-bounded set and
$c \| b \| \le \| b \xi \|$.
Let $a \in E_T$ be an arbitrary operator of norm $1$.
Since the matrix coefficients of $(a a^*)^n$ are located on $T^{\circ 2n}$,
there exists a unit vector $\xi \in \ell^2 (X)$
such that 
$\mathrm{supp}(\xi)$ is $S$-bounded and that
$c \| (a a^*)^n \| \le \| (a a^*)^n \xi \|$.
Since the norm of $(a a^*)^n$ is $1$, we have
\begin{eqnarray*}
\kappa^n < c \le  
\frac{\| (a a^*)^{n} \xi \|}{\| (a a^*)^{n-1} \xi \|}
\cdots
\frac{\| (a a^*)^{2} \xi \|}{\| (a a^*) \xi \|}
\frac{\| (a a^*) \xi \|}{\| \xi \|}.
\end{eqnarray*}
It follows that there exists $j = 0, 1, \cdots, n - 1$ such that
$\kappa < \| (a a^*)^{j + 1} \xi \| / \| (a a^*)^{j} \xi \| $.
By letting 
$\eta = a^* (a a^*)^{j} \xi / \| a^* (a a^*)^{j} \xi \|$,
we have the inequality
\begin{eqnarray*}
\kappa \| a \| 
= \kappa 
< \| (a a^*)^{j + 1} \xi \| / \| (a a^*)^{j} \xi \| 
\le \| a a^* (a a^*)^{j} \xi \| / \| a^* (a a^*)^{j} \xi \| 
= \| a \eta \|.
\end{eqnarray*}
The support of $a^* (a a^*)^{j} \xi$ is 
$(T^{\circ 2n - 1} \circ S)$-bounded.
Thus we obtain condition $(\beta)$ for $\kappa$, $T$, and $T^{\circ 2n - 1} \circ S$.
\end{proof}

To characterize ONL, we use `$E_T$' and `$\Phi_S$'
in the proof of Theorem \ref{TheoremNuclear}.

\begin{lemma}\label{LemmaONLbyNorm}
A uniformly finite coarse space $X$ has ONL
if and only if the following condition holds:
For every $\epsilon > 0$ and a controlled set $T$,
there exists a controlled set $S$ satisfying
$\left\| (\Phi_S |_{E_T})^{-1} \colon \Phi_S(E_T) \rightarrow E_T \right\| < 1 + \epsilon$.
\end{lemma}
\begin{proof}
Assume that $X$ has ONL.
For arbitrary $\epsilon > 0$ and a controlled set $T$, 
there exists $S$ which satisfies the condition $(\beta)$ 
for $c = (1 + \epsilon)^{-1}$ and $T$.

It follows that for every non-zero operator $a \in E_T$, 
there exists a unit vector $\eta \in \ell^2(X)$ 
whose support is $S$-bounded and satisfies
$\| a \| \le (1 + \epsilon) \| a \eta \|$.
Since $a \in E_T$,
$\mathrm{supp}(a \eta)$ is included in
the subset $T[\mathrm{supp}(\eta)]$.
Hence there exists a unit vector $\xi$ such that 
$\| a \eta \| = \langle a \eta, \xi \rangle$ and that 
the supports of $\xi$, $\eta$ are included in a common
$(T \circ S)$-bounded set.
There exists $x \in X$ satisfying
\begin{eqnarray*}
\| a \| \le (1 + \epsilon) \langle a \eta, \xi \rangle
\le (1 + \epsilon) \| [a_{y,z}]_{y, z \in T \circ S[x]} \| \le (1 + \epsilon) \| \Phi_{T \circ S} (a) \|,
\end{eqnarray*}
we get $\left\| (\Phi_{T \circ S} |_{E_T})^{-1} \right\| \le 1 + \epsilon$.

Conversely,
suppose that for every $c < 1$ and a controlled set $T$,
there exists a controlled set $S$ satisfying
$\left\| (\Phi_S |_{E_T})^{-1} \colon \Phi_S(E_T) \rightarrow E_T \right\| < 1/c$.
Then for every operator $a \in E_T$, 
there exists an $S$-bounded set $S[x]$ satisfying
\begin{eqnarray*}
c \| a \| \le \| [a_{y, z}]_{y, z \in S[x]} \|.
\end{eqnarray*}
Take a unit vector $\eta \in \ell^2(S[x])$ such that
$\| [a_{y, z}]_{y, z \in S[x]} \| = \| [a_{y, z}]_{y, z \in S[x]} \eta \|$.
The vector $\eta$ satisfies
$c \| a \| \le \| a \eta \|$.
It follows that condition $(\beta)$
holds true for $c <1$, $T$, and $S$.
We conclude that $X$ has ONL.
\end{proof}

Let us recall the notions of completely positive map and completely bounded map.
\begin{definition}
A  closed subspace $E$ of a unital C$^\ast$-algebra is
called an operator system, if
\begin{itemize}
\item
For every element $a$ of $E$, $a^*$ is also an element of $E$;
\item
The unit of the ambient C$^*$-algebra is an element of $E$.
\end{itemize}
A linear map $\Phi$ from $E$ to a C$^\ast$-algebra $C$
is said to be completely positive, if the map
$\Phi^{(n)} = \Phi \otimes \mathrm{id} 
\colon E \otimes \mathbb{M}_n(\mathbb{C}) \rightarrow C \otimes \mathbb{M}_n(\mathbb{C})$
is positive for every $n$.
\end{definition}

For controlled set $T \subset S$
The subspaces $E_T \subseteq C^*_u(X)$ 
and $\Phi_S(E_T) \subseteq D_S$ are examples of operator systems.
The map $\Phi_S \colon E_T \rightarrow C_S$ 
is completely positive.

\begin{definition}
Let $B$ be a C$^*$-algebra and let $F$ be an operator system.
A linear map $\Psi \colon F \rightarrow B$ is said to be completely bounded
if the increasing sequence $\{\|\Psi^{(n)} \colon 
F \otimes \mathbb{M}_n(\mathbb{C}) 
\rightarrow B \otimes \mathbb{M}_n(\mathbb{C}) \|\}$ is bounded.
The number $\| \Psi \|_{\rm cb} = \mathrm{sup}_{n \in \mathbb{N}} \| \Psi^{(n)}\|$ 
is called the completely bounded norm of $\theta$.
\end{definition}
The norms $\| \Psi \|$ 
and $\| \Psi \|_{\rm CB}$ are not identical in general, but
we have the following.

\begin{lemma}\label{LemmaONLbyCBNorm}
$\left\| (\Phi_S |_{E_T})^{-1} \colon \Phi_S(E_T) \rightarrow E_T \right\|_\mathrm{CB}
= \left\| (\Phi_S |_{E_T})^{-1} \right\|$.
\end{lemma}

\begin{proof}
For a natural number $n$, we denote by 
$((\Phi_S |_{E_T})^{-1})^{(n)}$ the linear map
\[
(\Phi_S |_{E_T})^{-1} \otimes \mathrm{id}: 
\Phi_S(E_T) \otimes \mathbb{M}(n, \mathbb{C})
 \rightarrow 
E_T \otimes \mathbb{M}(n, \mathbb{C}).
\]
It suffices to show that  
\begin{eqnarray*}
\| ((\Phi_S |_{E_T})^{-1})^{(n)} \| \le \| (\Phi_S |_{E_T})^{-1} \| \quad
\textrm{for every}\ n \in \mathbb{N}.
\end{eqnarray*}
Take an arbitrary positive number $K$ satisfying 
$K < \| ((\Phi_S |_{E_T})^{-1})^{(n)} \|$.
There exists an operator $a \in E_T \otimes \mathbb{M}(n, \mathbb{C})$
satisfying 
\[K < \| a \|, \quad \| \Phi_S \otimes \mathrm{id}_{\mathbb{M}(n, \mathbb{C})} (a) \| = 1.\]
We claim that there exist isometries
$V, W \colon \ell^2(X) \rightarrow \ell^2(X) \otimes \mathbb{C}^n$ satisfying
\begin{eqnarray*}
V \delta_x, W \delta_x \in \mathbb{C}\delta_x \otimes \mathbb{C}^n \quad 
\textrm{and} \quad
K < \| W^* a V \| \le \| a \|.
\end{eqnarray*}
Indeed, there exist unit vectors $\xi = \sum \delta_x \otimes \xi_x$ and $\eta = \sum \delta_y \otimes \eta_x$ such that
\[K < |\langle a \xi, \eta \rangle| < \|a\|.\]
We define isometries $V, W$ by
\begin{eqnarray*}
V(\delta_x) = 
\left\{
\begin{array}{ll}
\delta_x \otimes \xi_x / \| \xi_x \|, & \textrm{ if } \xi_x \neq 0,\\
\delta_x \otimes \delta_1, & \textrm{ if } \xi_x = 0,\\
\end{array}
\right.
\quad
W(\delta_x) = 
\left\{
\begin{array}{ll}
\delta_x \otimes \eta_x / \| \eta_x \|, & \textrm{ if } \eta_x \neq 0,\\
\delta_x \otimes \delta_1, & \textrm{ if } \eta_x = 0.\\
\end{array}
\right.
\end{eqnarray*} 
Then we have 
$K < |\langle a \xi, \eta \rangle| \le \|W^* a V\| \le \|a\|$.

Observe that the matrix coefficients 
of $W^* a V$ are zero on $X^2 \setminus T$ and
that 
\[
\| \Phi_S(W^* a V) \| \le \| \Phi_S^{(n)}(a) \| = 1.\]
It follows that $K < \| W^* a V \| / \| \Phi_S(W^* a V) \| \le \| (\Phi_S |_{E_T} )^{-1} \|$.
We obtain the inequality
$\| ((\Phi_S |_{E_T})^{-1})^{(n)} \| \le \| (\Phi_S |_{E_T})^{-1} \|$.
\end{proof}

\subsection{Property A and ONL}
The author proved in \cite[Theorem 4.1]{SakoONLP} that the operator norm localization property is equivalent to property A for a metric space with bounded geometry.  More generally, this theorem is also valid for uniformly locally finite coarse spaces. 
\begin{theorem}\label{TheoremONLP}
A uniformly locally finite coarse space $X$ has property A,
if and only if $X$ has the operator norm localization property.
\end{theorem}

\begin{proof}
We first assume that $X$ has property A.
Take an arbitrary controlled set $T \subset X^2$ and $\epsilon > 0$.
Define $\delta> 0$ by 
$\displaystyle{\frac{\epsilon}{\max\{ \sup \sharp(T[x]), \sup \sharp(T^{-1}[x])\}}}$.
By Theorem \ref{TheoremHulanicki-type}, there exist unit vectors 
$\{\eta_x\}_{x \in X} \subseteq \ell^2(X)$ satisfying
the following:
\begin{itemize}
\item
if $(x, y) \in T$, then $\|\eta_x - \eta_y\| < \epsilon$,
\item
$S = \{(x, z) \in X^2 ; z \in \mathrm{supp}(\eta_x)\}$ is controlled.
\end{itemize}
Recall that the C$^\ast$-algebra $C_S$ is defined by $\prod_{z \in X} \mathbb{B}(\ell^2(S[z]))$. We also use the unital completely positive maps $\Phi_S : C^*_\mathrm{u}(X) \rightarrow C_S$ and $\Psi : C_S \rightarrow C^*_\mathrm{u}(X)$ in the proof of Theorem \ref{TheoremNuclear}.
For an operator $b \in E_T$, and the matrix coefficient of $\Psi \circ \Phi (b)$ at $(x, y) \in T$ is $\langle \eta_y, \eta_x \rangle b_{x, y}$.
By Lemma \ref{LemmaNorm}, the following inequality follows:
\begin{eqnarray*}
\|\Psi \circ \Phi_S (b) - b\| 
&\le& 
\max\{\sup\sharp(T[x]), \sup\sharp(T^{-1}[x])\} \cdot
\sup_{(x, y) \in T} | \langle \eta_y, \eta_x \rangle b_{x, y} - b_{x, y} |\\ 
&\le& \epsilon \|b\|.
\end{eqnarray*}
Since $\Psi$ is contractive, we have
$\| \Phi_S (b) \| \ge \|\Psi \circ \Phi_S (b) \| \ge
(1- \epsilon) \|b\|$.
Then there exists $x \in X$ such that
$\| b |_{\ell^2(S[x])} \| \ge (1 - 2 \epsilon) \|b\|$.
It follows that there exists a unit vector $\xi \in \ell^2(S[x])$
such that $\| b \xi \| \ge (1 - 2 \epsilon) \|b\|$.
We conclude that $X$ has ONL.

Now assume that $X$ has ONL.
By Lemma \ref{LemmaONLbyNorm} and Lemma \ref{LemmaONLbyCBNorm}, 
for any controlled set $T$ and $\epsilon >0$, there exists a controlled set $S$ such that
\[
\| (\Phi_S |_{E_T})^{-1} \colon \Phi_S(E_T) \rightarrow E_T \|_{\rm CB} < 1 + \epsilon/2.
\]
It is easy to check that $(\Phi_S |_{E_T})^{-1}$ is unital and self-adjoint. 
By Corollary B.9 of the book \cite{OzawaBook}, 
there exists a unital completely positive map
$\Psi \colon B_S \rightarrow \mathbb{B}(\ell^2(X))$ which satisfies
$\| (\Phi_S |_{E_R})^{-1} - \Psi |_{\Phi_S (E_R)} \|_{\rm CB} < \epsilon$.

Define a function $k$ on the set $X^2$ by 
$k(x,y) = \langle \Psi \circ \Phi_S (e_{x,y}) \delta_y, \delta_x \rangle$,
where $e_{x, y}$ is the rank $1$ partial isometry which maps
$\delta_y$ to $\delta_x$.
Since $\Psi \circ \Phi_S$ is completely positive,
for every $x(1), x(2), \cdots, x(n) \in X$ and 
$\lambda_1, \lambda_2, \cdots, \lambda_n \in \mathbb{C}$,
we have
\begin{eqnarray*}
0
&\le&
\left\langle 
(\Psi \circ \Phi_S)^{(n)}
\left(
\left[
\begin{array}{cccc}
e_{x(1), x(1)} & \cdots & e_{x(1), x(n)} \\
\vdots & \ddots & \vdots \\
e_{x(n), x(1)} & \cdots & e_{x(n), x(n)} 
\end{array}
\right]
\right) 
\left[
\begin{array}{cccc}
\lambda_1 \delta_{x(1)}\\
\vdots\\
\lambda_n \delta_{x(n)}
\end{array}
\right] 
,
\left[
\begin{array}{cccc}
\lambda_1 \delta_{x(1)}\\
\vdots\\
\lambda_n \delta_{x(n)}
\end{array}
\right] 
\right\rangle\\
&=&
\sum_{i, j =1}^{n} \overline{\lambda_i} \lambda_j k(x(i), x(j)).
\end{eqnarray*}
It follows that $k$ is a positive definite kernel on $X$.
The support of $k$ is included in the controlled set $S$, because
$\Phi_S(e_{x, y}) = 0$ if $(x, y) \notin S$.
For $(x, y) \in T$, we have 
\begin{eqnarray*}
\left| 1 - k(x, y) \right| 
= \left| \langle (e_{x, y} - \Psi \circ \Phi_S (e_{x, y})) \delta_y, \delta_x \rangle \right| 
\le \left\| ((\Phi_S |_{E_R})^{-1} - \Psi)(\Phi_S (e_{x, y})) \right\| < \epsilon.
\end{eqnarray*}
It follows that $X$ satisfies condition (\ref{ConditionPositiveKernel}) in Theorem \ref{TheoremHulanicki-type}.
\end{proof}

\bibliographystyle{amsalpha}
\bibliography{remark.bib}

\end{document}